\documentclass[12pt,a4paper,reqno]{amsart}
\usepackage{amsmath,amscd,amssymb,latexsym}
\usepackage{longtable}

\newtheorem{Thm}{Theorem}[section]
\newtheorem{Cor}[Thm]{Corollary}
\newtheorem{Lem}[Thm]{Lemma}
\newtheorem{Prop}[Thm]{Proposition}
\theoremstyle{definition}
\newtheorem{Def}[Thm]{Definition}
\newtheorem{Exm}[Thm]{Example}

\newtheorem{Rk}[Thm]{Remark}

\def \ZZ{{\mathbb{Z}}}
\def \QQ{{\mathbb{Q}}}
\def \RR{{\mathbb{R}}}

\def \FF{{\mathbb{F}}}
\def \PP{{\mathbb{P}}}
\def \HH{{\mathbb{H}}}

\def \TTT{{\mathcal{T}}}

\def\Ell{\mathop{\mathrm{Ell}}}
\def\mod{\mathop{\mathrm{mod}}}

\def\vert{\mathop{\mathrm{vert}}}

\def\star{\mathop{\mbox{\Large$*$}}}
\def\Star{\mathop{\mbox{\LARGE$*$}}}
\def\Cl{\mathop{\mathrm{Cl}}}
\def\Cusp{\mathop{\mathrm{Cusp}}}

\def\Gamma{\varGamma}
\def\Delta{\varDelta}

\begin{document}

\begin{center}
\bf{\LARGE{On the automorphisms of the Drinfeld modular groups}}
\end{center}
\bigskip
\bigskip

\begin{center}
{A.~W.~MASON AND ANDREAS SCHWEIZER\footnote{The second author was supported by the Basic
Science Research Program through the National Research Foundation of Korea (NRF) funded by the Ministry
of Education (NRF-2022R1A2C1010487).}}
\end{center}

\title[Automorphisms of Drinfeld modular groups]{{}}

\address{A.~W.~Mason, 
Department of Mathematics, University of Glasgow,
Glasgow G12 8QW, Scotland, U.K.}
\email{awm@maths.gla.ac.uk}

\address{Andreas Schweizer,
Department of Mathematics Education, Kongju National University, 
Gongju 32588, South Korea}
\email{schweizer@kongju.ac.kr}

\begin{abstract}
Let $A$ be the ring of elements in an algebraic function field $K$ over $\mathbb{F}_q$ which are integral 
outside a fixed place $\infty$.  In contrast to the classical modular group $SL_2(\mathbb{Z})$ and the 
Bianchi groups, the {\it Drinfeld modular group} $G=GL_2(A)$ is not finitely generated and its automorphism 
group $\mathrm{Aut}(G)$ is uncountable. Except for the simplest case $A=\mathbb{F}_q[t]$ not much is 
known about the generators of $\mathrm{Aut}(G)$ or even its structure. We find a set of generators of 
$\mathrm{Aut}(G)$ for a new case.
\par
On the way, we show that {\it every} automorphism of $G$ acts on both, the {\it cusps} and the 
{\it elliptic points} of $G$. Generalizing a result of Reiner for $A=\mathbb{F}_q[t]$ we describe for each 
cusp an uncountable subgroup of $\mathrm{Aut}(G)$ whose action on $G$ is essentially defined on the 
stabilizer of that cusp. In the case where $\delta$ (the degree of $\infty$) is $1$, the elliptic points are 
related to the isolated vertices of the quotient graph $G\setminus\mathcal{T}$ of the Bruhat-Tits tree. 
We construct an infinite group of automorphisms of $G$ which fully permutes the isolated vertices with 
cyclic stabilizer.
\end{abstract}

\maketitle
\setcounter{section}{-1}

\section{Introduction.}\label{intro}

The introduction of Drinfeld modules \cite{Drinfeld} and Drinfeld modular curves has revolutionized 
the arithmetic theory of function fields. As function field analog of $\QQ$ one can consider any 
algebraic function field $K$ of one variable with a finite constant field $\FF_q$. The analog of $\ZZ$ 
then will be the (Dedekind) ring $A$ of elements of $K$ that are integral outside a chosen place $\infty$.
\par
Taking the place of the classical modular group $SL_2(\ZZ)$ and its action by M\"obius 
transformations on $\PP_1(\QQ)$, $\PP_1(\RR)$ and the complex upper halfplane $\HH$ is the 
{\it Drinfeld modular group} $G=GL_2(A)$ and its action by M\"obius transformations on $\PP_1(K)$,
$\PP_1(K_{\infty})$ and the Drinfeld upper halfplane $\Omega=C_\infty -K_\infty$. 
Here $K_\infty$ is the completion of $K$ at the place $\infty$, and $C_\infty$ is the completion of an 
algebraic closure of $C_\infty$.
\par
The simplest and best understood of these situations is the one where $K=\FF_q (t)$ is a rational
function field and $\infty$ the degree valuation, so $A=\FF_q [t]$.
\\ \\
In this paper we will concentrate on certain group theoretic aspects, more specifically, the 
automorphisms of such a modular group. For these questions the two theories indeed show the biggest 
difference, due to the following simple facts.
\par
The modular group $SL_2(\ZZ)$ (and the Bianchi groups $SL_2({\mathcal O}_d)$ where ${\mathcal O}_d$
is the ring of integers in an imaginay quadratic number field) are finitely generated. Moreover, their 
automorphisms groups are known \cite{HuaR}, \cite{SmillieVogt} and they form finite extensions of the 
group of inner automorphisms. By contrast, for Drinfeld modular groups $G$,

\begin{itemize}
\item $G$ is not finitely generated.
\item $\mathrm{Aut}(G)$ is uncountable and in particular not finitely generated.
\end{itemize}

The second property is of course only possible because the first one holds.
Both result from the fact that each stabilizer of a cusp contains an infinite-dimensional 
$\FF_q$-vector space. The size of $\mathrm{Aut}(G)$ represents a two-dimensional anomaly. From classical algebraic K-theory applied to arithmetic groups it is known \cite[4.3.15, p.175]{Hahn} that, when $n \geq 3$, $GL_n(A)$ is finitely generated so that then $\mathrm{Aut}(GL_n(A))$ is only countably infinite.
\par
These properties make finding a set of generators for $\mathrm{Aut}(G)$ a difficult task, and finding all
relation between them looks almost hopeless. Indeed, there is essentially only one case for which a set of
generators of $\mathrm{Aut}(G)$ is known \cite{Reiner}, namely $A=\FF_q[t]$ for any finite field $\FF_q$. 
In Section \ref{sect:applications} we will add another example to this list, this time with $q=2$ and 
$g=\delta=1$.
\par
But in general these properties suggest that problems that are slightly modified and look somewhat
easier might be more appropriate.
\begin{itemize}
\item[(1)] Find properties that all automorphisms of $G$ have.
\item[(2)] Find automorphism of $G$ that have a structural meaning.
\item[(3)] Find automorphisms that generate a huge subgroup of $\mathrm{Aut}(G)$.
\end{itemize}

\noindent
In \cite{MSquasi} we contributed to Question (2) by a thorough investigation of what we called the 
{\it quasi-inner automorphisms}. They are given by conjugating $G$ with elements from the normalizer 
of $G$ in $GL_2(K)$, and as such they automatically come equipped with an action as M\"obius transformations
on $K_{\infty}$, the Drinfeld upper half-plane $\Omega$, the Bruhat-Tits tree $\TTT$, and all related objects
like the Drinfeld modular curve, the quotient graph $G\setminus\TTT$, the cusps of $G$, the elliptic points.
But these automorphisms only form a finite extension of $\mathrm{Inn}(G)$ with $Quinn(G)$ (their quotient
by $\mathrm{Inn}(G)$) being isomorphic to $\mathrm{Cl}(A)_2$, the $2$-torsion in the ideal class group
of $A$ (see \cite{Cremona}).
\par
In the current paper we are working more in the direction of Questions (1) and (3). But a priori a general 
element of $\mathrm{Aut}(G)$ only comes with an action on group theoretic objects like subgroups of $G$ 
or conjugacy classes. So its action on other objects has to be defined indirectly from its action on the 
stabilizers of these objects.
\par
Here we recall that the cusps $\mathrm{Cusp}(H)$ of a finite index subgroup $H$ of $G$ are the
(finitely many) orbits under the action of $H$ on $\widehat{K}$. They are exactly the points that have 
to be filled in to make the quotient space $H\setminus\Omega$ into the $C_{\infty}$-analog of 
compact Riemann surface.
Moreover, $\mathrm{Cusp}(G)$ is in bijection with $\mathrm{Cl}(A)$, the ideal class group of the 
Dedekind ring $A$.
\par
In Section \ref{sect:cusps} we show that the stabilizer in $G$ of a point from $\PP_1(K)$ stabilizes \it only 
\rm that point, and hence the cusps of $G$ are in bijection with their stabilizers. So from the natural action 
of $\mathrm{Aut}(G)$ on these stabilizers we obtain the following result.

\begin{Thm}\label{intro1} 
$\mathrm{Aut}(G)$ acts on $\mathrm{Cusp}(G)$.
\end{Thm}

\noindent From Example \ref{Ex3cusps} we will see that in general this action is not transitive.
But as a Corollary we obtain that if $H$ is a finite index subgroup of $G$ and $\sigma$ any 
automorphism of $G$, then $H$ and $\sigma(H)$ have the same number of cusps.
\par
For elliptic points, which will be treated in 
Section \ref{sect:elliptic},
each stabilizer is cyclic of order $q^2 -1$. But the situation is slightly more complicated, as such
a stabilizer has $2$ fixed points, and the question is whether they give rise to the same elliptic point 
or not. Correspondingly, the set of elliptic points $\mathrm{Ell}(G)$ can be divided into two
subsets
$$\mathrm{Ell}(G)^{=} = \{G\omega:\omega \in E(G),\;G\omega=G\overline{\omega}\},$$
and
$$\mathrm{Ell}(G)^{\neq} = \{G\omega:\omega \in E(G),\;G\omega \neq G\overline{\omega}\}.$$
Luckily, the condition $G\omega=G\overline{\omega}$ is equivalent to the stabilizer $G^{\omega}$
being maximally finite in $G$. Since the action of $\mathrm{Aut}(G)$ on the stabilizers respects this 
group-theoretic property, we can conclude that

\begin{Thm}\label{intro3} 
$\mathrm{Aut}(G)$ acts on both $\mathrm{Ell}(G)^{=}$ and  $\overline{\mathrm{Ell}}(G)^{\neq}$
where
$$\overline{\mathrm{Ell}}(G)^{\neq} = \{\{G\omega,G\overline{\omega}\}:\omega \in E(G),\;G\omega \neq G\overline{\omega}\}.$$
\end{Thm}

In the next two sections we turn things around and use cusps and certain elliptic points to construct
big groups of automorphisms of $G$. 
\par
In Section \ref{sect:reiner} for every cusp we define an uncountable group of automorphisms.  We call
them {\it Reiner automorphisms} because the original idea for the construction, for the ring $A=\FF_q[t]$,
is in \cite{Reiner}. In \cite{MSJPAA} we had generalized these automorphisms to any $A$ but still only for 
the cusp $\infty$ (and used them to map congruence subgroups to non-congruence subgroups).
\par
In Section \ref{sect:spike} we pick up a result from \cite{MSMZ} again, namely that isolated vertices
of $G\setminus\TTT$ with cyclic stabilizer lead to a splitting of $G$ as a product amalgamated along 
the centre. This allows to define automorphisms separately on the factors of the product,  
provided one keeps the centre elementwise fixed. In the first step, Corollary \ref{Wreath} this leads to 
a wreath product that fully permutes the cyclic spikes. But surprisingly it also shows (Proposition 
\ref{notgraph}) that not every such automorphism induces an automorphism of the quotient graph
$G\setminus\TTT$.
\par
Inspired by some more classical results on automorphisms of free products, in Definition 
\ref{cyclicspikeauto} we extend this wreath product to an infinite group of automorphisms, which we 
call {\it cyclic spike automorphisms}.
\par
Using the description of the generators of $\mathrm{Aut}(GL_2(\FF_2[t]))$ from \cite{Reiner}, in
Section \ref{sect:applications} we finally obtain the following result.

\begin{Thm}
If $A=\FF_2[x,y]$ with $y^2 +y= x^3 +x+1$, then the inner automorphisms, Reiner automorphisms, 
and the cyclic spike automorphisms together generate $\mathrm{Aut}(G)$.
\end{Thm}

\noindent The following list of notations is compatible with the notation in \cite{MSMZ} and 
\cite{MSquasi}, except that in \cite{MSquasi} $\mathcal{V}_c$ is denoted by $\mathcal{V}$.
\\ \\
\begin{tabular}{ll}
$K$         & an algebraic function field of one variable with
constant field $\FF_q$, where $q=p^n$;\\
$g$ & the genus of $K$;\\
$\infty$        & a chosen place of $K$;\\
$\delta$        & the degree of the place $\infty$;\\
$\Omega$ & Drinfeld's upper half plane;\\
$A$          & the ring of all elements of $K$ that are integral outside $\infty$;\\
$G$             & the Drinfeld modular group $GL_2(A)$;\\
$Z$            & the centre of $G$ (consisting of the scalar matrices);\\
$B_2(R)$             & the subgroup of upper triangular matrices in $GL_2(R)$;\\
$U(A)$          & the subgroup of upper unipotent matrices in $G$;\\
$\widehat{K}$ & $\PP_1(K)=K \cup \{\infty\}$; \\
$\Cusp(G)$ & $G \backslash \widehat{K}$;\\
$G(s)$ & the stabilizer in $G$ of $s \in\PP_1(K)$;\\
$E(G)$ & the set of elliptic elements of $G$\\
$\Ell(G)$ & $G \backslash E(G)$;\\
$\Ell(G)^{=}$  & $\{G\omega:\omega \in E(G),\;G\omega=G\overline{\omega}\}$;\\
$\Ell(G)^{\neq}$  &  $\{G\omega:\omega \in E(G),\;G\omega \neq G\overline{\omega}\}$;\\
\end{tabular}
\\ \\
\begin{tabular}{ll}
$G^{\omega}$&  the stabilizer in $G$ of $\omega \in E(G)$; \\
$\mathcal{T}$ & the Bruhat-Tits tree of $G$;\\
$G_v$ & the stabilizer in $G$ of $v \in \vert(\mathcal{T})$;\\
$r$      & $=|\mathcal{V}_c|$, the number of cyclic spikes of $G\setminus\TTT$ (if $\delta=1$).
\end{tabular}

 \section{General properties of automorphisms: cusps} \label{sect:cusps}

\noindent It is well-known that $A$ is a Dedekind domain and that $A^*=\FF_q^*$. For each $s \in \widehat{K}$ we record some well-known properties of the stabilizer $G(s)$. (See, for example, \cite{MasonEdin}.)  The subset of unipotent matrices in $G(s)$ is a non-trivial subgroup which we denote by $U(s)$.  

\begin{Lem}\label{only}  
The only element of $\widehat{K}$ fixed by any non-trivial element of $U(s)$ is $s$.
\end{Lem} 

\begin{Lem}\label{list}
There exists an integer $n_s \geq 1$ such that, for all $n \geq n_s$, there exists a fiinite subgroup $G_n(s)$ (resp. $U_n(s)$) of $G(s)$ (resp. $U(s)$) for which
\begin{itemize}
\item[(i)] $$\bigcup_{n \geq n_s}G_n(s)=G(s)\;and\;\bigcup_{n \geq n_s}U_n(s)=U(s).$$\\ 
\item[(ii)] $G_n(s) \leq G_{n+1}(s)$ and $U_n(s) \leq U_{n+1}(s)$.\\ 
 \item[(iii)]$[G_n(s), G_n(s)]= U_n(s)\;(q>2)$ and $G_n(s)=U_n(s)\;(q=2)$.\\ 
 \item[(iv)]$[G(s), G(s)]= U(s)\;(q>2)$ and $G(s)=U(s)\;(q=2)$.\\ 
\item[(v)]  $U_n(s) \cong \FF_{q^n}^+$.\\ 
\item[(vi)] $G_n(s) /U_n(s) \cong\FF_q^*\times \FF_q^*.$
\end{itemize}
\end{Lem}

\noindent \rm We now come to the principal result in this section. 

\begin{Thm}\label{rational points}  
Let $\sigma \in \mathrm{Aut}(G)$. Then, for all $s \in \widehat{K}$,
$$\sigma(G(s))=G(s'),$$
\noindent for some $s' \in \widehat{K}$.
\end{Thm} 

\begin{proof} 
\rm Consider any finite subgroup $G_n(s)$ as in Lemma \ref{list}, where $n\geq n_s$. Then $\sigma(G_n(s)) \leq G_v$, for some $ v \in \vert(\TTT)$ by \cite[Proposition 2, p.76]{Serre}. There are two possibilities. If $G_v$ contains a matrix with eigenvalues {\it not} in $k$ then
$$G_v \cong \FF_{q^2}^*\;\; \mathrm{or}\;\; GL_2(\FF_q),$$ 
\noindent by \cite[Corollary 2.12]{MSstabilizer}.
\noindent We conclude then that every matrix in $G_v$ has eigenvalues in $k$. In which case, by the proofs of \cite[Theorems 2.1, 2.3]{MSstabilizer}, there  exists $g \in GL_2(K)$ and a {\it finite} subset $\mathcal{S}$ of $K$ for which
  $$g^{-1}G_vg= \left\{\left[\begin{array}{cc}\alpha&c\\0&\beta
\end{array}\right]: \alpha,\beta \in \FF_q^*,\;c\in \mathcal{S}\right\}.$$
\noindent It follows that
$$\sigma(G_n(s)) \leq G(s'),$$
\noindent where $s'=g(\infty)$. By an identical argument there exists $s''$ such that
$$(\sigma(G_n(s)) \leq) \sigma(G_{n+1}(s)) \leq G(s'').$$
\noindent Assume that $q>2$. Then by Lemma \ref{list} (iv)
$$\sigma(U_n(s)) \leq U(s') \cap U(s''),$$
\noindent in which case $s'=s''$ by Lemma \ref{only} and so
$$\sigma(G(s)) \leq G(s'),$$
\noindent by Lemma \ref{list} (i).
\noindent By an identical argument
$$\sigma^{-1}(G(s')) \leq G(s_0),$$
\noindent for some $s_0$. By the first part $G(s) \leq G(s_0)$ in which case $s=s_0$ again by Lemma \ref{list} (iv) and Lemma \ref{only} and so $\sigma(G(s))=G(s')$. The proof for the case $q=2$ is simpler. 
\end{proof}

\noindent \begin{Cor}\label{unipotent}
$$\sigma(U(s))=U(s').$$ 
\end{Cor}

\begin{proof}  
\rm Follows from Lemma \ref{list} (iv).
\end{proof}

\noindent It follows that $\mathrm{Aut}(G)$ acts as a group of permutations on $\widehat{K}$.  

\begin{Cor} \label{perm}  
For each $\sigma \in \mathrm{Aut}(G)$ and $s \in \widehat{K}$  let
$$\widetilde{\sigma}: s\; \rightarrow\;s',$$
where $\sigma(G(s))=G(s')$ as above. Then $\widetilde{\sigma}$ is a well-defined permutation of $\widehat{K}$.\end{Cor}

\begin{proof} 
By Lemma \ref{only} and Lemma \ref{list} (iv) $\widetilde{\sigma}$ is both well-defined and injective. 
For surjectivity apply the same argument with $\sigma$ replaced by $\sigma^{-1}$.
\end{proof}

\noindent Since $A$ is a Dedekind ring it is well-known that there exists a one-to-one correspondence
$$\mathrm{Cusp}(G)\longleftrightarrow \mathrm{Cl}(A),$$
\noindent where $\mathrm{Cl}(A)$ is the ideal class group of $A$. In addition, since $A$ is an {\it arithmetic} domain, $\mathrm{Cusp}(G)$ is {\it finite}. For each $s \in  \widehat{K}$ we put
$$ G[s]= \left\{ G(g(s))=gG(s)g^{-1}: g \in G\right\}.$$ 
\\
\noindent By Theorem \ref{rational points} the action of any $\sigma \in\mathrm{Aut}(G)$ on any $G(s)$ readily extends to $G[s]$ and so the action  of $\mathrm{Aut}(G)$ can be extended to $\mathrm{Cusp}(G)$. The following are immediate consequences of Lemma \ref{only} and Corollary \ref{perm}. It is clear that
$$G[s_1]=G[s_2]\;\mathrm{ if \;and\; only \;if} \;s_2=g(s_1), \mathrm{for \;some}\;g \in G.$$
\noindent \begin{Lem}\label{cusps}  For any $\sigma \in \mathrm{Aut}(G)$ and $s \in \widehat{K}$
$$  \sigma (G[s])=G[\widetilde{\sigma}(s)].$$
\end{Lem}

\begin{proof} 
Follows from Corollary \ref{perm}.
\end{proof}

\noindent  Now let $\{s_1,\cdots,s_n\} \subseteq \widehat{K}$ be a complete set of representatives for $\mathrm{Cusp}(G)$, where $n = |\mathrm{Cl}(A)|$.  It is clear from above then that 
$$\left\{G[s]: s \in \widehat{K}\right\}=\left\{ G[s_1], \cdots , G[s_n]\right\}\leftrightarrow\Cusp(G).$$ 

\begin{Thm}\label{permcusp}  
Let $\sigma \in \mathrm{Aut}(G)$. For each $i \in \{1, \cdots ,n \}$, let 
$$\sigma^*(G[s_i])=G[s_j],$$
where $s_j=g(\widetilde{\sigma}(s_i))$, for some $g \in G$. Then
$$\sigma^*:\Cusp(G) \rightarrow \Cusp(G),$$
is a well-defined permutation of $\Cusp(G)$.
\end{Thm}

\begin{proof} 
Follows from Corollary \ref{perm} and Lemma \ref{cusps}.
\end{proof}

\noindent \rm It is clear that inner automorphisms, for example, act trivially on $\Cusp(G)$.   
\\ \\
In Example \ref{Ex3cusps} we will see that in general the action of $Aut(G)$ on the cusps is not
transitive. 
\\ \\
\noindent We conclude this section with an interesting property which is common to all automorphisms of $G$. Let $H$ be a finite index subgroup of $G$.  Then $\Cusp(H)=H \backslash \widehat{K}$ is finite. We denote its order by $c(H)$.

\begin{Thm}\label{cuspnumber} 
Let $\sigma \in \mathrm{Aut}(G)$. Then, for all finite index subgroups $H$ of $G$,
$$ c(\sigma(H))=c(H).$$
\end{Thm}

\begin{proof} 
It is well-known that
\begin{displaymath} c(H)= \sum_{i=1}^n |H\backslash G / G(s_i)|, \end{displaymath}

\noindent where each term in the sum is the number of double cosets of $H$ and $G(s_i)$ in $G$.
Clearly $|H \backslash G / G(s_i)|=|\sigma(H) \backslash G/ \sigma(G(s_i))|$. The result follows from the proof of Theorem \ref{permcusp} since $\{\sigma(s_1), \cdots ,\sigma(s_n)\}$ is another complete set of representatives for $\Cusp(G)$.
\end{proof}

\begin{Exm} 
We recall \cite[Section 2]{MSquasi} that a \it non-trivial quasi-inner automorphism \rm of $G$ is a map
$$\iota_ g \rightarrow ghg^{-1}\;(h \in G),$$
where $g \in N_{\widehat{G}}(G) \backslash G.Z_K$ and $Z_K$ is the set of all scalar matrices in $\widehat{G}=GL_2(K)$. It is also known that
$$Quinn(G):=N_{\widehat{G}}(G) / G.Z_K \cong \Cl(A)_2,$$
the $2$-torsion in $\Cl(A)$, the ideal class group of $A$.
\noindent For this section the relevant result  \cite[Corollary 6.2]{MSquasi} is
$$Quinn(G) \;acts\; freely\;on\;\Cusp(G).$$ 
\end{Exm}

\section{General properties of automorphisms: elliptic points} \label{sect:elliptic}

\noindent \rm  Let $\omega \in E(G)$. Then, by definition, $\omega \in  \Omega$ and $G^{\omega}$, 
its stabilizer in $G$, is {\it non-trivial} i.e. it contains non-scalar matrices. It is known that then 
$\omega\in\widetilde{K}-K$ where $\widetilde{K}=\FF_{q^2}.K$ \cite[Lemma 2.2]{MSMZ} and 
$G^{\omega} \cong\FF_{q^2}^*$ \cite[Proposition 2.3]{MSMZ}.
Since by definition $\omega \notin K_{\infty}$, where $K_{\infty}$ is the completion of $K$ with 
respect to $\infty$, it follows that \cite[Corollary 2.4]{MSMZ} $E(G) \neq \emptyset$ if and only if 
$\delta$ is \it odd. 
\rm We will assume throughout this section, unless otherwise stated, that $\delta$ is odd. 
\par
Note that if $\omega$ is elliptic, then $\overline{\omega}$, the image of $\omega$ under 
$Gal(\widetilde{K}/K)$, is also elliptic with $G^{\overline{\omega}}=G^{\omega}$.
Clearly $G$ acts on $E(G)$.  
\\ \\
\noindent {\bf Definition.}  
For each $\omega \in E(G)$ let
$$ G\omega= \left\{g(\omega): g \in G \right\}.$$ 
The \it elliptic points \rm of $G$ are defined to be the elements of the set
$$\Ell(G)= G \backslash E(G)=\left\{ G\omega: \omega \in  E(G) \right\}.$$ 
\noindent It is well known that $\Ell(G)$ is finite.

\begin{Thm}\label{polynomial} 
$$|\Ell(G)|=L_K(-1),$$
\it where $L_K(u)$ is the $L$-polynomial of $K$. 
\end{Thm}

\begin{proof} 
See, for example, \cite[Corollary 3.6]{MSMZ}. 
\end{proof}

\begin{Rk} 
We note that $GL_2(\FF_q) \leq G$, for \it any \rm $\delta$. Let $\epsilon$ be any generator of $\FF_{q^2}^*$. We put $\lambda=\epsilon\overline{\epsilon}$ and $\mu=\epsilon+\overline{\epsilon}$.
\noindent Consider the following elements of $GL_2(\FF_q)$.
$$g=\left[\begin{array}{cc}0&\lambda\\-1&\mu\end{array}\right]\;\;\mathrm{and}\;\; g'=\left[\begin{array}{cc}0&\lambda\\1&0
\end{array}\right].$$\noindent Then $G^{\epsilon}=G^{\overline{\epsilon}}=<g>$ and $g'(\epsilon)=\overline{\epsilon}$. Then $\epsilon,\overline{\epsilon}$ are elliptic elements of $G$ when (and only when) $\delta$ is odd. In general then (when $\delta$ is odd) $|\Ell(G)| \geq 1$. This bound is best possible. Consider for example the case where $g(K)=0$ and $\delta$ is odd. It is known \cite[Theorem 5.1.15, p.193]{Stich} then that $L_K(u)=1$.  
\end{Rk}

\noindent As with $\Cusp(G)$ we determine an action of $\mathrm{Aut}(G)$ on the elliptic points via its action on their stabilizers. We record some basic properties of these subgroups.

\begin{Lem}\label{elist} 
Let $\omega,\;\omega_1,\;\omega_2 \in E(G)$ and $g \in G$. The following are obvious.
\begin{itemize}
\item[(i)] $G^{\omega}=G^{\overline{\omega}}$.
\item[(ii)] $gG^{\omega}g^{-1}=G^{g(\omega)}$.
\item[(iii)] $G^{\omega_1}=G^{\omega_2}$ if and only if 
$\left\{\omega_1,\overline{\omega_1}\right\}=\left\{\omega_2,\overline{\omega_2}\right\}$.
\end{itemize}
\end{Lem}

\begin{Def} Let
$$\mathcal{C}=\left\{C\leq G: C,\;\mathrm{ cyclic \;order} \;q^2-1\right\}.$$
and
$$\overline{E}(G)=\{ \{\omega,\overline{\omega}\}: \omega \in E(G)\}.$$
(The conjugate pairs are unordered.) 
\end{Def}
 
\noindent The starting point for our study of the action of $\mathrm{Aut}(G)$ on $\Ell(G)$ is its 
natural action on $\mathcal{C}$. 

\begin{Lem}\label{bijcyclic} \cite[Lemma 2.6]{MSMZ} 
If $\delta$ is odd, mapping $\{\omega,\overline{\omega}\}$ to $G^{\omega} =G^{\overline{\omega}}$ 
induces a natural bijection 
$$\overline{E}(G)\longleftrightarrow \mathcal{C}.$$
The inverse map is given by mapping the cyclic subgroup to its two fixed points. 
\end{Lem}

\begin{Thm}\label{acts3} 
$\mathrm{Aut}(G)$ acts on $\overline{E}(G)$.
\end{Thm}

\begin{proof} 
Follows from Lemma \ref{bijcyclic}. Let $\{\omega,\overline{\omega}\} \in \overline{E}(G)$ and let $\sigma \in \mathrm{Aut}(G)$. Then $\sigma(G^{\omega})=G^{\omega_0}$, for some unique $\{\omega_0, \overline{\omega_0}\} \in \overline{E}(G)$. We define $\overline{\sigma}: \overline{E}(G) \rightarrow 
\overline{E}(G)$ by
$$\overline{\sigma}(\{\omega,\overline{\omega}\})=(\{\omega_0, \overline{\omega_0}\}).$$
\end{proof}

\noindent
To refine this we need \cite[Lemma 4.1]{MSMZ}, which says that under the building map $\lambda$ from
$\Omega$ to the Bruhat-Tits tree $\TTT$ any elliptic element $\omega$ maps to a vertex $v=\lambda(\omega)$
of $\TTT$ and $G^{\omega}\leq G_v$. Moreover, $\omega$ and $\overline{\omega}$ map to the same vertex.
\par
More precisely we have

\begin{Prop}\label{poss}  \cite[Proposition 4.4]{MSMZ}
Let $\omega\in\Omega$ be an elliptic element, and let $v=\lambda(\omega)$ be the vertex to which it 
maps under the building map. There are two possibilities.
\begin{itemize} 
\item[(a)] If $G\omega=G \overline{\omega}$, then 
$$\FF_{q^2}^* \cong G^{\omega} \lneqq G_v \cong GL_2(\FF_q).$$
\item[(b)]  If $G\omega \neq G\overline{\omega}$, then
$$G^{\omega}= G_v \cong \FF_{q^2}^*.$$
\end{itemize}
\end{Prop}

This has been strengthened in \cite[Lemma 5.8]{MSquasi}, which says that $G_v\cong\FF_{q^2}^*$ if and only
if $G^{\omega}$ is maximally finite. Consequently we define
 
\begin{Def} Let
$$\mathcal{C}_{mf} =\left\{C \in \mathcal{C}: C \hbox{\rm \ is maximally finite}\right\}$$
and $\mathcal{C}_{nm}=\mathcal{C}-\mathcal{C}_{mf}$.
\end{Def}

Clearly, the action of $\mathrm{Aut}(G)$ respects these group-theoretic properties.
 
\begin{Thm}\label{acts1} 
$\mathrm{Aut}(G)$ acts on both $\mathcal{C}_{mf}$ and $\mathcal{C}_{nm}$. 
\end{Thm}

Now we consider the induced actions on the $G$-orbits. 

\begin{Def} 
(a) For each subgroup $S$ of a group $T$ we put
$$S^T = \left\{S^t=tSt^{-1}: t\in T\right\}.$$ 
(b) Let $$(\mathcal{C}_{nm})^G =\left\{C^G : C \in \mathcal{C}_{nm} \right\}\;\mathrm{and}\;\;(\mathcal{C}_{mf})^G=\left\{C^G: C \in \mathcal{C}_{mf}\right\}.$$
\end{Def}

\begin{Cor}\label{acts2}  
$\mathrm{Aut}(G)$ acts on both $(\mathcal{C}_{nm})^G$ and $(\mathcal{C}_{mf})^G$.
Actually, the action factors through the action of $\mathrm{Aut}(G)/\mathrm{Inn}(G)$.
\end{Cor} 

\begin{Def} 
We partition $\Ell(G)$ as follows:
\\ \\
(i) $\Ell(G)^{=}=\left\{G\omega: \omega \in E(G),G\omega=G\overline{\omega}\right\}$ \\ \\
(ii) $\Ell(G)^{\neq}=\left\{G\omega: \omega \in E(G),G\omega\neq G\overline{\omega}\right\}$ \\ \\
(iii) $ \overline{\Ell}(G)^{\neq}=\left\{\{G\omega,G\overline{\omega}\}: \omega \in E(G),G\omega\neq G\overline{\omega}\right\}$.
\end{Def}

\noindent It is clear from the remarks after Theorem \ref{polynomial} that $\Ell(G)^{=} \neq \emptyset$. We can refine previous results as follows.

\begin{Thm}\label{surjective} 
The map in Lemma \ref{bijcyclic} induces 
\begin{itemize}
\item[(i)] a one-to-one correspondence
$$\Ell(G)^{=} \longleftrightarrow (\mathcal{C}_{nm})^G.$$
\item[(ii)]  and when $\Ell(G)^{=} \neq \Ell(G)$ a one-to-one correspondence
$$\overline{\Ell}(G)^{\neq}  \longleftrightarrow (\mathcal{C}_{mf})^G.$$
\item[(iii)]Note that there exists a two-to-one surjective map 
$$\Ell(G)^{\neq} \longrightarrow \overline{\Ell}(G)^{\neq}$$
\end{itemize}
\end{Thm}

\begin{Cor}\label{acts4}  
$\mathrm{Aut}(G)$ acts on both $\Ell(G)^{=}$ and $\overline{\Ell}(G)^{\neq}$.
Actually, the action factors through the action of $\mathrm{Aut}(G)/\mathrm{Inn}(G)$.
\end{Cor}

In \cite[Theorem 4.9]{MSquasi} we showed that $Quinn(G)$ acts freely and transitively on
$\Ell(G)^=$. 
If $\delta=1$, in Section \ref{sect:spike} we will construct a group of automorphisms that fully permutes
$\overline{\Ell}(G)^{\neq}$.

\section{Reiner automorphisms}\label{sect:reiner}

\noindent The automorphisms described in this section derive from a decomposition of $G$ as a free amalgamated due to Serre \cite[Theorem 10, p.119]{Serre}.  We make use of a version of this result which refers explicitly to matrices \cite[Theorem 4,7]{MasonTAMS}. The simplest case i.e. $A=\FF_q[t]$ of this result was first proved by Nagao \cite[Theorem 6, p.86]{Serre}. Automorphisms of this type were first introduced by Reiner \cite{Reiner} for the case $A=\FF_q[t]$. We begin with a more detailed version of Lemma \ref{list}. We recall 
\cite[ Proposition 2, p.76]{Serre} that every vertex stabilizer is finite.

\begin{Lem}\label{vertexstabilizer} 
For each $s \in \widehat{K}$ there exists an infinite half-line $\mathcal{L}$ in $\mathcal{T}$, where $\vert(\mathcal{L})=\{v_1, v_2, \cdots\}$ such that
\begin{itemize}\item[(i)] $$\bigcup_{i=1}^{\infty} G_{v_i}=G(s),$$
\item[(ii)] $G_{v_i} \leq G_{v_{i+1}}$\;($i \geq 1$).
\end{itemize}
\end{Lem}

\begin{proof} 
See \cite[Section $4$]{MasonTAMS}.
\end{proof}

\noindent It is clear that
$$G(\infty)=G(0)^T=\left\{\left[\begin{array}{cc}\alpha&a\\0&\beta\end{array}\right] : \alpha,\beta \in \FF_q^*,\; a \in A \right\}. $$
\noindent  Suppose now that $s \neq 0,\infty$. Let
$$ M_s=\left[\begin{array}{cc}s&1\\1&0\end{array}\right].$$ Then $$ X \in G(s) 
\Longleftrightarrow 
X=\left[\begin{array}{cc}\alpha+cs&d\\c&\beta-sc\end{array}\right]=M_s\left[\begin{array}{cc}\beta&-c\\0&\alpha\end{array}\right]M_s^{-1},$$
where $\alpha ,\beta \in \FF_q^*, \; c \in A\cap As^{-1}$ and $\mathrm{det}(X)=\alpha\beta$. See \cite[Theorem 2.1]{MasonEdin}. We denote $X$ by $[\alpha,\beta,c]$. Let $\mathfrak{q}_s$ denote the $A$-ideal $A\cap As^{-1}\cap As^{-2}$. Then
$$U(s)=\left\{u(c)=[1,1,c]:\; c \in \mathfrak{q}_s\right\},$$
\noindent The homomorphism $\psi: G(s) \rightarrow \FF_q^* \times \FF_q^*$ defined by
$$\psi(X)=(\alpha,\beta)$$
\noindent is \it surjective \rm (since $A$ is a Dedekind domain). See \cite[Corollary 3.2]{MasonEdin}. Let $S$ be a subset of $G(s)$, where $|S|=(q-1)^2$, for which $\psi(S)=\FF_q^* \times \FF_q^*$. Then from Lemma \ref{vertexstabilizer} there exists $i$ such that
$$\psi(G_{v_i})=\FF_q^* \times \FF_q^*.$$
\noindent Let $i_0$ be the smallest such integer. 

\begin{Lem}\label{finite} 
For all $s \in \widehat{K}$ there exists a finite subgroup $F$ of $G(s)$ for which 
\begin{itemize}
\item[(i)] $\psi(F)=\FF_q^* \times \FF_q^*.$ \
\item[(ii)] $G(s)=F.U(s).$
\end{itemize}
\end{Lem}

\begin{proof} 
When $s=0,\infty$ the result is obvious. When $s \neq 0,\infty$  we may take, for example,  $F=G_{v_{i_0}}$. Part (ii) follows immediately from part (i).
\end{proof}
\noindent When $s \neq 0,\infty$ it is clear that $U(s) \cong \mathfrak{q}_s^+$, the additive group of the $\FF_q$-space 
$\mathfrak{q}_s$. (It is clear that $U(0) \cong U(\infty) \cong A^+$.) Let $V$ be any subspace of $\mathfrak{q}_s$. Then $G(s)$ acts by conjugation on $U(V)=\{u(v): v \in V\}$ as scalar multiplication. Let $F \cap U(s)=U(V_0)$, where $V_0$ is a \it finite-dimensional \rm subspace and let $\mathfrak{q}_s= V_0 \oplus V^*$. Then
 $$G(s)=F.U(V^*)\;\mathrm{and} \;F \cap U(V^*)=1.$$

\noindent  This is also true for $s=0,\infty$. We can state the following without proof.

\begin{Lem} \label{auto1} 
Let $s \in \widehat{K}$. With the above notation, let $\phi: V \rightarrow V$ be a $\FF_q$-linear map for which 
\begin{itemize} 
\item[(i)] $\phi$ acts as the identity on $V_0$,
\item[(ii)] $\phi$ acts as an automorphism of $V^*$.
\end{itemize} 
\noindent Let $f \in F$ and $v \in V^*$. Then the map $\phi(=\phi_{s,F}):G(s) \rightarrow G(s)$ defined by
$$ \phi(fu(v))=fu(\phi(v)),$$
\noindent is an automorphism of $G(s)$.
\end{Lem}

\begin{Def}
A \it cuspidal ray \rm in a graph $\mathcal{G}$ is an infinite half-line without backtracking all of whose vertices have  valency $2$ with the exception of its terminal vertex.
\end{Def}

\begin{Thm}\label{amalgam1} Let $s \in \widehat{K}$. There exists a subgroup $H \leq G$ and a finite subgroup $J \leq G(s)\cap H$ such that
$$G=G(s) \star_{\quad J} H.$$
\end{Thm}

\begin{proof} Serre \cite[Theorem 9, p.106]{Serre} has proved that the quotient graph has the following structure
$$ G \backslash \TTT=Y \cup\left(\bigcup_{1 \leq i \leq n}\mathcal{R}_i\right),$$
where $Y$ is finite and each $\mathcal{R}_i$ is a cuspidal ray. In addition the elements of the set $\left\{ \mathcal{R}_i \right\}$ are in one-one correspondence with those of $\Cusp(G)= G \backslash \widehat{K}$, so that $n=|\Cl(A)|$. See \cite[pp. 104-106]{Serre}. \\
\noindent A presentation for $G$ is then derived from a lift of a maximal tree of $G \backslash \TTT$ to $\TTT$ 
by \cite[Theorem 13, p.55]{Serre}. It can be shown \cite[Section $4$]{MasonTAMS} that, for any $s \in \widehat{K}$, some $\mathcal{R}_i$ lifts to an infinite half-line $\mathcal{L}$ in $\TTT$ as described in Lemma \ref{vertexstabilizer}. Then
$$G=G(s) \star_{\quad J} H,$$
where $J=G_{v_j}$, for some $j$.
\end{proof}

\noindent An immediate consequence is one of the principal results in this section.

\begin{Cor} \label{autocusp} 
With the above notation let $j_0 \geq j$ be the smallest integer for which $F=G_{v_{j_0}}$ maps onto $\FF_q^* \times \FF_q^*$. Now let $\phi(=\phi_{s,F}) \in \mathrm{Aut}(G(s))$ as in Lemma \ref{auto1}.
Let
$$\phi^*(g)=\left\{\begin{array}{ccc}\phi(g) & , & g \in G(s) \\ g & , & g \in H \end{array} \right.$$
Then $\phi^*$ extends to an automorphism of $G$.
\end{Cor}

\begin{proof} 
Follows from the normal form theorem for free amalgamated products. The inverse of $\phi^*$ is defined in the obvious way.
\end{proof}

\begin{Rk} (i) Automorphisms of this type were first introduced by Reiner 
\cite{Reiner} for the simplest case $A=\FF_q[t]$ in \cite{Reiner}. He refers to such automorphisms as \it non-standard. \rm \\ \\
(ii) The special case $s=\infty$ is already known \cite[Theorem 2.4]{MSJPAA}. \\ \\
(iii) Theorem \ref{amalgam1} was first proved for the case $A=\FF_q[t]$ by Nagao. See \cite[Theorem 6, p.86]{Serre}. In this case we may take $J=B_2(\FF_q[t])$.
\end{Rk}
\noindent Let $\rho_s$ denote the automorphism described in Corollary  \ref{autocusp}. We record one important property.

\begin{Cor}\label{fixcusp} Under the action described in Theorem \ref{permcusp} each automorphism $\rho_s$ fixes every element of $\Cusp(G)$.
\end{Cor}

\begin{Def} 
In view of \cite{Reiner} we refer to $\rho$ as a \it Reiner automorphism. \rm  
\end{Def}

\begin{Thm}\label{uncountable} 
Let $\mathcal{R}_s$ be the set of all Reiner automorphisms $\rho_s$. Then
$$\mathrm{Card}(\mathrm{Aut}(G))=\mathrm{Card}(\mathcal{R}_s) =2^{\aleph_0}.$$
\end{Thm}

\begin{proof}  
Clearly it suffices to prove this for the case $s=\infty$. The proof is based on an approach introduced in \cite{MSJPAA}. Let $T(a)=I_2+E_{12}\;(a \in A)$ and for each $A$-ideal $\mathfrak{q}$ let
$$\Gamma(\mathfrak{q})=\{ X \in SL_2(A):X \equiv I_2 \;(\mod\;\mathfrak{q}) \}.$$
By Corollary \ref{autocusp} for the special case $\tau=\rho_{\infty}$ restricted to $\Gamma(\mathfrak{q})$ it follows that, for almost all $\mathfrak{q}$, there exist uncountable sets $\{V_{\lambda}: \lambda \in \Lambda \}$ and $\{\tau^{\lambda}: \lambda \in \Lambda\}$, where each $\tau^{\lambda}$ is a Reiner automorphism and $V_{\lambda}$ is a finite codimensional $\FF_q$-subspace of $A$, for which, under the one-one correspondence $\tau^{\lambda} \leftrightarrow V_{\lambda}$,
$$ \{ a \in A: T(a) \in \tau^{\lambda}(\Gamma(\mathfrak{q}))\}= V_{\lambda}.$$
There are then uncountably many finite index subgroups $\tau^{\lambda}(\Gamma(\mathfrak{q}))$. The proof follows.
\end{proof}

\section{Cyclic spike automorphisms}\label{sect:spike}

\noindent For the sake of clarity it is convenient to alter some 
of the previous notation. Recall \cite[Theorem 5.1]{MSstabilizer} that a vertex $\tilde{v} \in \vert(G\backslash \TTT)$ is \it isolated \rm if and only if $\delta =1$ and $G_v \cong GL_2(\FF_q)\; \mathrm{or}\;\FF_{q^2}^*$. We assume throughout this section that $\delta=1$. We define the \it non-cyclic \rm vertices
$$\mathcal{V}_{nc}=\{ \tilde{v}\in \vert(G\backslash\TTT): G_v \cong GL_2(\FF_q) \}$$ 
and the \it cyclic \rm vertices
$$\mathcal{V}_c=\{ \tilde{v}\in \vert(G\backslash\TTT): G_v \cong \FF_{q^2}^* \}.$$ 
\noindent It is known \cite[Remark 3) p.97]{Serre}
that $\mathcal{V}_{nc} \neq \emptyset$ (for \it any \rm $\delta$). On the other hand  the example of Nagao's theorem \cite[1.6, p.85]{Serre} shows that, for the case $A=\FF_q[t]$, $\mathcal{V}_c = \emptyset$.
\noindent By \cite[Lemma 4.2(ii), Theorem 3.13]{MSquasi}  we have the following one-one correspondences$$\Ell(G)^= \longleftrightarrow \Cl(A)_2 \longleftrightarrow \mathcal{V}_{nc}.$$ 
\noindent In addition a consequence of \cite[Theorem 4.8]{MSMZ}
is the following.
$$\overline{\Ell}(G)^{\neq}=\{\{G\omega,G\overline{\omega}\}:G\omega \neq G\overline{\omega}\} \longleftrightarrow \mathcal{V}_c \longleftrightarrow (\mathcal{C}_c)^G.$$
\noindent Given that $|\overline{\Ell}(G)^{\neq}|=\frac{1}{2}|\Ell(G)^{\neq}|$ we have from a previous equation
$$|\Ell(G)|=L_K(-1)=|\Cl(A)_2|+2|\mathcal{V}_c|.$$

\begin{Def} We refer to  an isolated vertex in a graph together with its incident edge as a \it spike\rm.
\end{Def}

\begin{Lem}\label{elementwise} 
Let
$$\mathcal{A}=\mathrm{Aut}(\FF_{q^2}^* / \FF_q^*)=\{ \sigma \in \mathrm{Aut}(\FF_{q^2}^*): \sigma \;fixes\; \FF_q^*\; elementwise \}.$$

\noindent Then
$$|\mathcal{A}|=\left\{\begin{array}{rcl} 2\varphi(q+1) & \mbox{for} & q,\; odd\\ \varphi(q+1) & \mbox{for} &  q,\;even \end{array} \right .$$
\end{Lem}

\begin{proof} 
In the absence of a suitable reference we sketch the proof. If $\sigma \in \mathrm{Aut}(\FF_{q^2}^*)$ then $g: \lambda \rightarrow \lambda^a$, where $(a,q^2-1)=1$. Is clear that $\sigma \in \mathcal{A}$ if and only if $q-1$ divides $a-1$. Then $|\mathcal{A}|$ is the number of $u\;(\mod \;q^2-1)$, where (i) $( u,q+1)=1$ and (ii) $u\equiv 1\;(\mod \;q-1)$. 
\noindent Suppose that $q$ is odd. Then $(q+1,q-1)=2$. Let $a$ belong to a fixed complete reduced set of residues $(\mod\;q+1)$. By an extended version of the Chinese Remainder Theorem there exists $u$ such that (i) $u \equiv\:a\;(\mod\;q+1)$ and (ii) $u \equiv 1\;(\mod\; q-1)$ and $u$ is unique modulo $\frac{q^2-1}{2}$, the least common multiple of $q\pm 1$.
The $\varphi(q+1)$ choices extend to $2\varphi(q+1)$ modulo $q^2-1$.
When $q$ is even $(q+1,q-1)=1$ and this case follows using a similar argument.
\end{proof}

\noindent Let $r=|\mathcal{V}_c|$.

\begin{Lem}\label{amalgam2} 
Suppose that $\delta=1$ and that $r>0$. There exists subgroups $C_i$ of $G$ $(0 \leq i \leq r)$ such that $G$ is the following free amalgamated product
$$G=\Star _{\quad Z} C_i,$$
where \begin{itemize}
\item[(i)] $Z=\{\lambda I_2:\lambda \in \FF_q^*\}$ is the centre of $G$,
\item[(ii)] $C_i \cong \FF_{q^2}^*$ when $i >0$.
\end{itemize}
\end{Lem}

\begin{proof} 
By the \it fundamental theorem of the theory of groups acting on trees \rm \cite[Theorem 13 p.55]{Serre} $G$ is the \it fundamental group of a graph of groups \rm \cite[p.42]{Serre} arising from a \it lift \rm of a maximal tree in $G \backslash \TTT$. Any such lift must contain isolated vertices which map onto those of $\mathcal{V}_{nm}$. Let $e$ be any edge of the lift which maps onto the edge (in $G \backslash \TTT$) incident with one of the vertices in $\mathcal{V}_c$. Then $G_e=Z\; (\cong \FF_q^*$) by \cite[Theorem 4.1]{MSstabilizer}. The proof follows.
\end{proof}

\noindent It is clear that $\{C_1, \cdots,C_{r}\}$ is a complete set of representatives for $\{(\mathcal{C}_c)^G\}$.\;For each $k>0$ it is known by the proof of \cite[Theorem 2.6]{MSstabilizer} that there exists $x_k \in GL_2(C_{\infty})$ such that
$$C_k=\left\{x_k \mathrm{diag}(\lambda,\sigma(\lambda))x_k^{-1}: \lambda \in \FF_{q^2}^*\right\}\cong \FF_{q^2}^*.$$ Since the fixed point set of the Frobenius map $\lambda\mapsto \lambda^q$
is precisely $\FF_q$ it is clear that $C_k$ contains $Z=\{\lambda I_2 :\lambda \in \FF_q^*\}$. The map
$$x_i \mathrm{diag}(\lambda,\sigma(\lambda))x_i^{-1} \mapsto x_j \mathrm{diag}(\sigma(\lambda),\lambda)x_j^{-1}$$ defines an isomorphism
$$\psi_{ij}: C_i \rightarrow C_j$$
which fixes $\FF_{q^2}^*$.
\noindent It is clear that for each $i,j$ there are $|\mathcal{A}|$  such isomorphisms. 
Clearly $\psi_{ji}=\psi_{ij}^{-1}$.

\begin{Thm}\label{autocyclic} 
With the above notation suppose that $ij>0$. There exists an element $\psi^{i\leftrightarrow j}(=\psi^{j\leftrightarrow i})\in \mathrm{Aut}(G)$ for which

$$\psi^{i\leftrightarrow j}(g)=\left\{\begin{array}{ccc}\psi_{ij}(g) & , & g \in C_i \\ \psi_{ji}(g) & , & g \in C_j \\g & , & g \in C_k \end{array} \right.$$
for all $k \neq i,j$. 
\end{Thm} 

\begin{proof} 
Follows from Lemma \ref{amalgam2} by the normal form theorem. 
\end{proof}

\begin{Def} 
For the case $i=j (\neq 0)$ the same argument shows that there exists $\psi^i \in \mathrm{Aut}(G)$ which restricts to an automorphism of $C_i$ and fixes all other $C_k$. 
\\ 
\noindent Let $\mathrm{CS}(G)$ be the subgroup of $\mathrm{Aut}(G)$ generated by all automorphisms
$\psi^{i\leftrightarrow j}$ and $\psi^i$.
\end{Def}

\noindent The identification of (i) each $\psi^i$ with the identity in $S_r$ and (ii) each $\psi^{i \leftrightarrow j}$ with the transposition $(i,j)$ extends to a natural map $\pi: \mathrm{CS}(G) \rightarrow S_t$ which in turn leads to the following description of $\mathrm{CS}(G)$.

\begin{Cor}\label{Wreath} 
With the above notation suppose that $r=\frac{1}{2}(L_K(-1)-|\Cl(A)_2|)>0$. Then $\mathrm{CS}(G)$ is isomorphic to the wreath product  
$$R\wr S_r,$$
where $R 
\cong \mathcal{A}$. Consequently
$$|\mathrm{CS}(G)|=r!|\mathcal{A}|^r.$$
\end{Cor}

\begin{proof} 
Clearly $\pi$ is surjective. Let $V_i=\{ \psi^i\}$, where $1\leq i \leq r$. Then $V_i \cong \mathcal{A}$. From the above $\mathrm{CS}(G)$ acts on the elements of set $\{V_1, \cdots V_r\}$ in the process permuting them (fully) in accordance with $S_r$. The result follows. The order of $\mathrm{CS}(G)$ is then given by Lemma \ref{elementwise}.
\end{proof}

\begin{Rk} 
By definition $\mathrm{CS}(G)$ acts on the isolated vertices of $G\backslash \TTT$ of cyclic type while fixing all other vertices. If a  pair of such vertices are attached to $G\backslash \TTT$ at a single vertex the corresponding  automorphism extends naturally to an automorphism of $G \backslash \TTT$. This restriction is necessary.
\end{Rk}

\begin{Prop}\label{notgraph}
Suppose that $\delta=1$, $q\geq 8$  and $\frac{3^g}4 > q+1$ (so in particular $g>3$).  
Then $\mathrm{CS}(G)$ contains an automorphism that does {\bf not} induce an automorphism 
of the quotient graph $G\setminus\TTT$.
\end{Prop}

\begin{proof} 
As above let $r$ denote the number of isolated vertices of cyclic type in $G \backslash \TTT$. Then, for 
$q \geq 8$ and $g>3$,
it is known \cite[Theorem 5.6(b)]{MSMZ} that $r > \frac{3^g}4$. Now the valency of every vertex of $\TTT$ is $q+1$. The condition ensures then that there are at least two such isolated vertices which are attached to the rest of $G \backslash \TTT$ by different vertices. From the above the corresponding cyclic spike automorphism has the required property.
\end{proof}

The condition that $K$ has a place of degree $1$ (to take as infinite place) is not automatic. But it also is
not an overly restrictive condition. For every $q$ and $g$ there are function fields with such a place.
\par
The following is inspired by the main result of \cite{Fouxe}. See Theorem \ref{generators} below for more 
context.

\begin{Lem}\label{partialconjugate}
Let $G$ be an amalgamated product
$$G=\Star_Z H_i$$
with $0\leq i \leq m$. 
Pick distinct $i$ and $j$ from $\{0,\ldots,m\}$ and $g_i \in H_i$. 
Then the {\it partial conjugation} $\alpha_{ij}$ defined by
$$\alpha_{ij}(h):=\left\{\begin{array}{ccl} g_ihg_i^{-1} & \hbox{\it if} & h\in H_j, \\ h & \hbox{\it if} & h\in H_k\ \hbox{\it with}\ k\neq j,\end{array}\right.$$
is an automorphism of $G$.
\end{Lem}

\begin{proof}
Since the conjugation fixes the elements from the centre, one easily sees that $\alpha_{ij}$ defines 
a group endomorphism. Its inverse is the partial conjugation with $g_i^{-1}$.
\end{proof}

\begin{Def}\label{cyclicspikeauto}
Assume that the number $r$ of cyclic spikes is positive. Let $\overline{\mathrm{CS(G)}}$ be the infinite overgroup of $\mathrm{CS(G)}$
generated by the wreath product $\mathrm{CS}(G)$ together with all partial conjugations  (see Lemma \ref{partialconjugate}) 
of the amalgamated product
$$G=\Star_Z H_i$$
from Lemma \ref{amalgam2}.
\end{Def}

\noindent In the next section we will see some concrete examples of this as well as some very subtle points 
one has to be extremely careful with.

\section{Some applications}\label{sect:applications}

There is only one case for which a set of generators of $\mathrm{Aut}(G)$ is known, namely $A=\FF_q[t]$ for any finite field $\FF_q$ (see \cite[Theorem, p.465]{Reiner}). 
\par
\noindent Before pushing a little bit further in this direction we state some results that we will need.

\begin{Lem}\label{notfree} 
 
(a) $GL_2(\FF_q[t]) = GL_2(\FF_q)\star_{B_2(\FF_q)}B_2(\FF_q[t])$ (Nagao's Theorem) 
\begin{center} (b)  $GL_2(\FF_q[t])$ does not decompose as a non-trivial free product.
\end{center}
\end{Lem}

\begin{proof} For (a) see \cite[Corollary, p.85]{Serre}.
\par
\noindent For (b) assume to the contrary that  $G=GL_2(\FF_q[t])=H_1\star H_2$, where $H_1$ and $H_2$ are non-trivial.  Now any free product must have trivial centre and so we may assume further that $q=2$. In this case $B_2(\FF_2[t])$ is an abelian torsion group and so by the Kurosh Subgroup Theorem is contained in a conjugate of one of the factors, say $H_1$. Let $N$ denote the normal subgroup of $G$ generated by $H_1$. Then $G/N \cong H_2$. But $N$ contains $GL_2(\FF_2)$ and hence $G$. The proof follows.
\end{proof}

\noindent If $q>2$, then in order to generate $\mathrm{Aut}(GL_2(\FF_q[t]))$ one needs further types of automorphisms that we have not discussed in this paper. See \cite[Theorem, p.465]{Reiner}.
But if $q=2$, this significantly simplifies matters because then $\FF_q^* =\{1\}$ and the identity matrix 
is the only diagonal matrix. We will make use of the following.

\begin{Prop}\label{Reinerq=2}
Let $q=2$, $g=0$ and $\delta=1$, i.e. $A=\FF_2[t]$ and $G=GL_2(\FF_2[t])$. Then every automorphism 
of $G$ is of the form $\sigma\tau$ where $\sigma$ is an inner automorphism and $\tau$ is a Reiner automorphism.
\end{Prop}

\begin{proof}
Let $\alpha$ be an arbitrary automorphism of $G=GL_2(\FF_2[t])$.
\noindent To start with, by \cite[Corollary 3.14]{MSquasi} every subgroup of $G$ that is isomorphic to $GL_2(\FF_2)$
actually is conjugate in $G$ to the natural subgroup $GL_2(\FF_2)$ of $GL_2(A)$ coming from $\FF_2\subset A$.
So we can find an inner automorphism $\sigma_1$ of $G$ such that $\sigma_1\alpha$ restricted to 
$GL_2(\FF_2)$ is an automorphism of $GL_2(\FF_2)$. 
\par
Secondly, every automorphism of $GL_2(\FF_2)$ is inner. So there exists an element of $GL_2(\FF_2)$ that
gives an inner automorphism $\sigma_2$ of $G$ such that $\sigma_2\sigma_1\alpha$ restricted to 
$GL_2(\FF_2)$ is the identity. In particular, $\sigma_2\sigma_1\alpha$ fixes the matrix
$M={1\ 1\choose 0\ 1}$.
\par
Next, $U(A)=B_2(\FF_2[t])$ is the centralizer of $M$, and therefore $\sigma_2\sigma_1\alpha(U(A))$ is the centralizer of
$\sigma_2\sigma_1\alpha(M)=M$; so $\sigma_2\sigma_1\alpha(U(A))=U(A)$.  
\par
Since $U(A)$ is a direct 
sum of countably infinitely many copies of $\FF_2$ we can apply a Reiner automorphism $\tau_1$ to achieve that $\tau_1\sigma_2\sigma_1\alpha$ is the identity on $U(A))$ (and also still on $GL_2(\FF_2)$).
As $U(A)$ and $GL_2(\FF_2)$ generate $G$, this means that $\tau_1\sigma_2\sigma_1\alpha$ is the 
identity on $G$.
\end{proof}

\noindent Alternatively, we could argue that, as a special case of \cite[Theorem]{Reiner} for $q=2$, 
$\mathrm{Aut}(GL_2(\FF_2[t]))$ is generated by inner and Reiner automorphisms. (The other automorphisms
needed to generate $\mathrm{Aut}(GL_2(\FF_q[t]))$ for $q>2$ do not exist for $q=2$.) Since the inner 
automorphisms are normal in $\mathrm{Aut}(G)$ any product of both can be written in the above form.
\\ \\
\noindent
For the remaining examples we assume $g=\delta=1$, the \it elliptic \rm case. Here $A=\FF_q[x,y]$, where $x,y$ satisfy a \it Weierstrass equation. \rm Takahashi \cite{Takahashi} has determined the structure of $G \backslash \TTT$ for \it all \rm fields of constants. In all cases he has shown (i)  $G\backslash \TTT$ is a \it tree \rm and (ii) $G \backslash \TTT$ has a vertex with (trivial) stabilizer $Z$. By the theory of groups acting on trees
\cite[Theorem 13, p.55]{Serre} $G$ (resp. $PGL_2(A))$ is a free product (resp. amalgamated free product) when $q=2$ (resp. $q>2$). 

\begin{Thm}\label{generators} \cite{Fouxe}
Let $G=A_1\star A_2 \star\cdots\star A_n$ be a free product where the factors $A_i$ cannot be split further 
as free products and no factor $A_i$ is infinite cyclic. Then a system of generators of the automorphism group 
$\mathrm{Aut}(G)$ is given by all automorphisms of the following types (in the numbering of \cite{Fouxe}).
\\ \\
\noindent (1) Maps $\varphi_i$ which restrict to an element of $\mathrm{Aut}(A_i)$ and restrict to the identity 
on $A_j$, where $j \neq i$. \\
\noindent (2) Maps $\alpha_{ij}$ which restrict to the identity on $A_k$, where $k \neq j$. For some  $a_i \in A_i$,
$$ \alpha_{ij}(a_j)=a_i a_j a_i^{-1},$$
 for all $a_j \in A_j$. \\
 
\noindent (8) For each pair $(i,j)$ for which there is an isomorphism $\omega$ from $A_i$ to $A_j$,  the map 
$\omega_{ij}$ which interchanges $A_i$ and $A_j$  
and restricts to the identity on $A_k$ according to
$$\omega_{ij}(g)=\left\{
\begin{array}{lcl}\omega(g) & \hbox{\it if} & g\in A_i,\\
\omega^{-1}(g) & \hbox{\it if} & g\in A_j,\\
g & \hbox{\it if} & g\in A_k \hbox{\it\ with\ } k\neq i,j.\\
\end{array}\right.$$
\end{Thm}

\begin{proof} 
This is the content of the Theorem on page 268 of \cite{Fouxe}. Moreover, the same paper on 
page 269 gives a set of relations $(12)$ to $(21)$, which by the Theorem on \cite[p.275]{Fouxe} 
generate all relations in $\mathrm{Aut}(G)$.
\par
For a slightly less opaque approach to these results which is based on the action of $G$ on a certain 
simplicial complex $K(G)$ see \cite{MemoirAMS}.
\par
As for the automorphisms of type (8), it suffices to take one for each pair $(i,j)$ with $A_i \cong A_j$; 
the other ones are obtained by composing it with automorphisms of type (1).
\end{proof}

\begin{Rk}\label{Rkconjugate}
One question immediately comes to mind: Why not conjugate $A_j$ with an arbitrary element $g$ of $G$?
The answer is that this would be an injective group endomorphism of $G$, but in general {\it not} surjective.
\end{Rk}

We explain this with the simplest theoretically possible example.

\begin{Exm}\label{dihedral}
Consider
$$D_\infty =\langle a\ |\ a^2 =1\rangle\star\langle b\ |\ b^2 =1\rangle=\langle ab\rangle\rtimes\langle b\rangle.$$
Conjugating $\langle a\rangle$ with $ab$ we obtain
$$\langle ababa\rangle \star \langle b\rangle = \langle (ab)^3\rangle\rtimes\langle b\rangle,$$
which is a subgroup of index $3$. \\

\noindent A group which is not isomorphic to any of its proper subgroups is called \it co-Hopfian. \rm
\end{Exm}

\begin{Exm}\label{Ex3cusps}
Let $A=\FF_2[x,y]$, where $$y^2+y=x^3.$$
This elliptic curve has $3$ $\FF_2$-rational points, namely $\infty$, $(0,0)$ and $(0,1)$.  
(The $L$-polynomial of this elliptic function field is $L_K(u)=1+2u^2$.) 
\par
Using the main results and the notation from \cite{Takahashi} we obtain the quotient graph 
$G\setminus\TTT$ as follow.
\\ \\ 
{\unitlength0.5cm
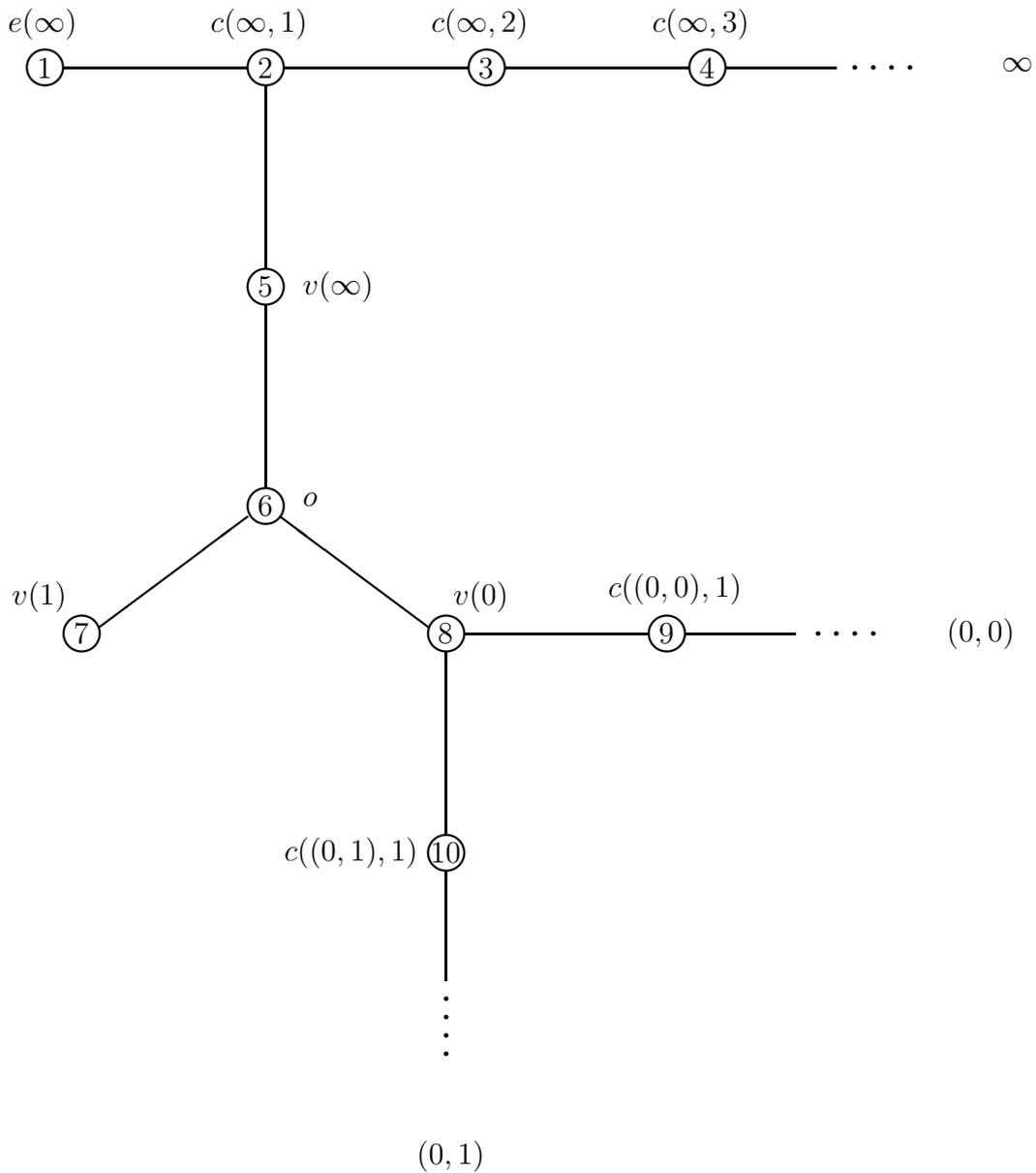
\begin{picture}(0,36)
\thicklines
\put(9,34){\circle{1.0}\makebox(-2,0){2}}
\put(3,34){\circle{1.0}\makebox(-2,0){1}}
\put(15,34){\circle{1.0}\makebox(-2,0){3}}
\put(21,34){\circle{1.0}\makebox(-2,0){4}}
\put(9,28){\circle{1.0}\makebox(-2,0){5}}
\put(9,22){\circle{1.0}\makebox(-2,0){6}}

\put(4,18.5){\circle{1.0}\makebox(-2,0){7}}
\put(13.9,18.5){\circle{1.0}\makebox(-2,0){8}}

\put(3.5,34){\line(1,0){5}}
\put(14.5,34){\line(-1,0){5}}
\put(15.5,34){\line(1,0){5}}
\put(9,33.5){\line(0,-1){5}}
\put(21.5,34){\line(1,0){3}}
\put(9,27.5){\line(0,-1){5}}

\put(4.5,18.7){\line(4,3){4}}
\put(13.4,18.7){\line(-4,3){4}}

\put(2,35){$e(\infty)$}
\put(7.5,35){$c(\infty,1)$}
\put(13.5,35){$c(\infty,2)$}
\put(19.5,35){$c(\infty,3)$}
\put(10,27.8){$v(\infty)$}
\put(10,22){$o$}
\put(2.1,19.3){$v(1)$}
\put(14.1,19.3){$v(0)$}

\put(25,34){\circle{.1}}\put(25.5,34){\circle{.1}}\put(26,34){\circle{.1}}\put(26.5,34){\circle{.1}}

\put(19.9,18.5){\circle{1.0}\makebox(-2,0){9}}
\put(13.9,12.5){\circle{1.0}\makebox(-2,0){10}}
\put(14.4,18.5){\line(1,0){5}}
\put(13.9,18){\line(0,-1){5}}
\put(20.4,18.5){\line(1,0){3}}
\put(24,18.5){\circle{.1}}\put(24.5,18.5){\circle{.1}}\put(25,18.5){\circle{.1}}\put(25.5,18.5){\circle{.1}}
\put(13.9,12){\line(0,-1){3}}
\put(13.9,8.5){\circle{.1}}\put(13.9,8){\circle{.1}}\put(13.9,7.5){\circle{.1}}\put(13.9,7){\circle{.1}}
\put(18.3,19.5){$c((0,0),1)$}
\put(9.5,12.3){$c((0,1),1)$}
\put(29,33.9){$\infty$}
\put(27.5,18.3){$(0,0)$}
\put(13.1,4){$(0,1)$}

\end{picture}}

\noindent The structure of the vertex stabilizers for the above can be inferred from \cite[Theorem 5]{Takahashi}. 
\begin{itemize}
\item[(i)] the stabilizer of $e(\infty)$, $S(e(\infty))=GL_2(\FF_2)$ ($\FF_2 \subset A$)
\item[(ii)] $S(c(\infty,1))=S(v(\infty))=\langle{1\ 1\choose 0\ 1}\rangle$, 
\item[(iii)] for $n \geq 2,\; S(c(\infty,n))=\{{1\ a\choose 0\ 1}:\mathrm{deg}(a) \leq n\}$.
\item[(iv)] $S(o)$ is trivial,
\item [(v)] $S(v(1))$ is cyclic order $3$,
\item[(vi)] all edges adjacent to $c(\infty,1)$ have stabilizer $\langle{1\ 1\choose 0\ 1}\rangle$,
\item[(vii)] all edges adjacent to $o$ have trivial stabilizer,
\item[(viii)] $S(v(0))$ is trivial, and hence all edges attached to $v(0)$ have trivial stabilizer,
\item[(ix)] $S(c((0,0),n))$ is an $\FF_2$-vector space of dimension $n$,
\item[(x)] $S(c((0,1),n))$ is an $\FF_2$-vector space of dimension $n$.
\end{itemize}
\noindent In (iii) the \it degree \rm of $a$ refers to its (function field) degree as an element of a $A$.
\par
From the fact that the central vertex $o$ and the vertex $v(0)$ both have trivial stabilizer we obtain 
a splitting of $G$ as a free product
$$G=H\star \langle M_1 \rangle \star A_0 \star A_1$$
where $M_1$ has order $3$, 
$H=GL_2(\FF_2)\star_{B_2(\FF_2)}B_2(A)\cong GL_2(\FF_2[t])$, and 
$A_0$ and $A_1$ (the stabilizers of the cusps $(0,0)$ and $(0,1)$) are $\FF_2$-vector spaces of 
countably infinite dimension.
\par
Now we can apply Theorem \ref{generators} to this free product and obtain a set of generators of 
$\mathrm{Aut}(G)$.
\par
In particular, this shows that the action of $\mathrm{Aut}(G)$ on the cusps has $2$ orbits, one consisting 
of $(0,0)$ and $(0,1)$ and one consisting of the cusp $\infty$. So the action of $\mathrm{Aut}(G)$ on 
the cusps described in Theorem \ref{permcusp} is in general not transitive.
\end{Exm}

\begin{Rk}\label{importantremarks}

\noindent  (a) The quotient graph $G\setminus\TTT$ does not really know the stabilizer of a vertex, only
the conjugacy class of that stabilizer. Take Example \ref{Ex3cusps}. The maximal finite subgroups of order
$3$ are precisely the conjugates of $\langle M_1\rangle$. But once we choose the upper triangular matrices 
as the stabilizer of the cusp $\infty$, the matrix $M_1$ is more or less fixed (not just up to conjugacy).  
This is the meaning of lifting a maximal tree in $G\setminus\TTT$ to $\TTT$. From a different point of view 
we have just seen that we cannot replace $M_1$ by an arbitrary conjugate of it if we want the free product 
to be the full group $G$.
\\

\noindent (b) Similar care is of course required in the slightly more general situations of 
Lemma \ref{amalgam2} and Theorem \ref{amalgam1}.
\\

\noindent (c) Another point we have seen in Example \ref{Ex3cusps} is that the group $C_0$ from 
Lemma \ref{amalgam2} might split further as a free product. In general we have no a priori knowledge 
about this.
\end{Rk}

\begin{Exm}\label{Ex1cusp}
Let $A=\FF_2[x,y]$, where $$y^2+y=x^3+x+1.$$
This is the unique ellliptic case over $\FF_2$ with only one cusp.  Equivalently, $A$ is a principal ideal domain.
(The $L$-polynomial of this elliptic function field is $L_K(u)=1-2u+2u^2$.) Serre \cite[2.4.4, p.115]{Serre}, using 
the theory of vector bundles, has determined the shape of $G\setminus\TTT$. Alternatively, this can again be 
obtained from \cite{Takahashi}.
\\ \\ \\
{\unitlength0.5cm
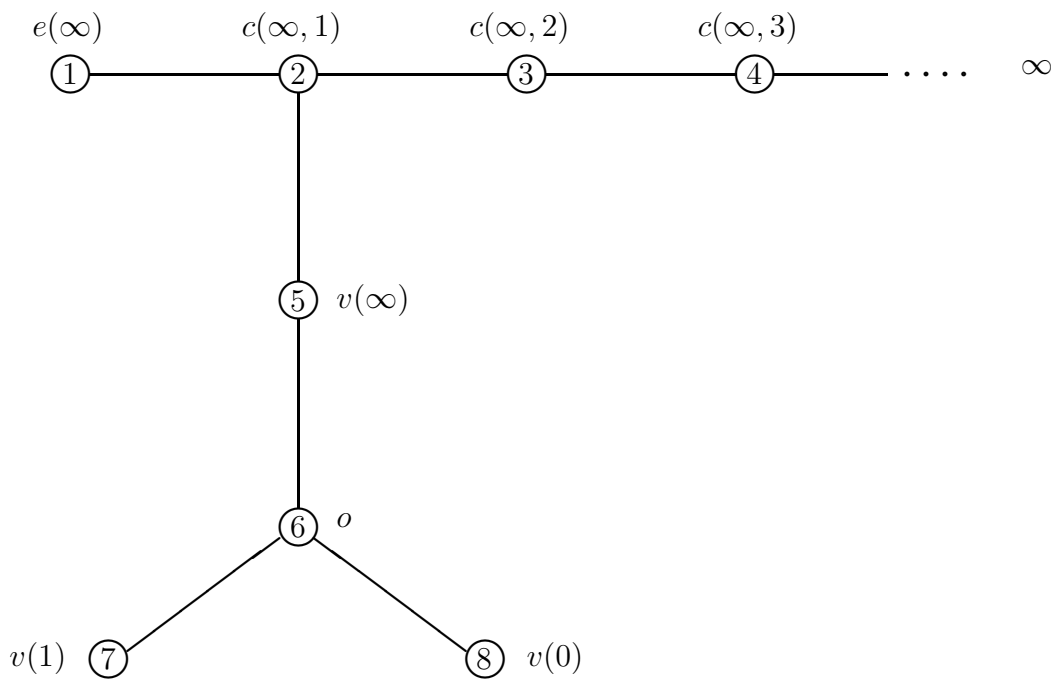
\begin{picture}(0,21)
\thicklines
\put(9,19){\circle{1.0}\makebox(-2,0){2}}
\put(3,19){\circle{1.0}\makebox(-2,0){1}}
\put(15,19){\circle{1.0}\makebox(-2,0){3}}
\put(21,19){\circle{1.0}\makebox(-2,0){4}}
\put(9,13){\circle{1.0}\makebox(-2,0){5}}
\put(9,7){\circle{1.0}\makebox(-2,0){6}}

\put(28,19){$\infty$}

\put(4,3.5){\circle{1.0}\makebox(-2,0){7}}
\put(13.9,3.5){\circle{1.0}\makebox(-2,0){8}}

\put(3.5,19){\line(1,0){5}}
\put(14.5,19){\line(-1,0){5}}
\put(15.5,19){\line(1,0){5}}
\put(9,18.5){\line(0,-1){5}}
\put(21.5,19){\line(1,0){3}}
\put(9,12.5){\line(0,-1){5}}

\put(4.5,3.7){\line(4,3){4}}
\put(13.4,3.7){\line(-4,3){4}}

\put(2,20){$e(\infty)$}
\put(7.5,20){$c(\infty,1)$}
\put(13.5,20){$c(\infty,2)$}
\put(19.5,20){$c(\infty,3)$}
\put(10,12.8){$v(\infty)$}
\put(10,7){$o$}
\put(1.4,3.3){$v(1)$}
\put(15,3.3){$v(0)$}

\put(25,19){\circle{.1}}\put(25.5,19){\circle{.1}}\put(26,19){\circle{.1}}\put(26.5,19){\circle{.1}}

\end{picture}}

The quotient graph $G\setminus\TTT$ looks like the picture of the graph in Example \ref{Ex3cusps}, 
except that vertex $8$ now is a terminal vertex, i.e. the cusps $(0,0)$ and $(0,1)$ are completely removed 
(including the vertices $9$ and $10$ and the edges attached to them). Also, the stabilizer of vertex 
$8$ now is cyclic of order $3$.
\par
By the same arguments as in Example \ref{Ex3cusps} we obtain
$$ G=H\star\langle M_0\rangle \star\langle M_1\rangle,$$
where $H\cong GL_2(\FF_2)\star_{U(\FF_2)} U(A) \cong GL_2(\FF_2[t])$ and
$\langle M_0 \rangle$ and $\langle M_1\rangle$ are two non-conjugate groups of order $3$.
So again we can apply Theorem \ref{generators} to describe $\mathrm{Aut}(G)$.
However, for this special case it is also possible to provide a set of generators for $\mathrm{Aut}(G)$ 
without using \cite{Fouxe}.
\end{Exm}

\begin{Thm}\label{Aut(G)}
Let $A=\FF_2[x,y]$ with
$$y^2 +y=x^3 +x+1.$$
Then the following automorphisms of $G=GL_2(A)$
\begin{itemize}
\item[{(a)}] the inner automorphisms of $G$;
\item[{(b)}] the Reiner automorphisms $\rho_\infty$;
\item[{(c)}] the cyclic spike automorphisms $\overline{\mathrm{CS}(G)}$;
\end{itemize}
together form a set of generators for $\mathrm{Aut}(G)$.
\end{Thm}
\begin{proof}
The wreath product described in Corollary \ref{Wreath} contains the automorphisms that switch
$\langle M_0 \rangle$ and $\langle M_1 \rangle$ as well as any automorphism of $\langle M_i \rangle$
extended by the identity on the other two free factors. By Definition \ref{cyclicspikeauto} the cyclic spike 
automorphisms also
contain the automorphisms that are the identity on two of the three free factors while the remaining 
factor is conjugated with an element from the other two. So we still need the automorphisms of $H$
extended by the identity on $\langle M_0 \rangle\star\langle M_1 \rangle$.
\par
From $H\cong GL_2(\FF_2[t])$ and Proposition \ref{Reinerq=2} we see that every automorphism of $H$ 
has the form $\sigma\tau$ where $\tau$ is a Reiner automorphism and $\sigma$ is conjugation with
an element $h$ from $H$. The Reiner automorphism $\tau$ acts as identity on 
$\langle M_0 \rangle\star\langle M_1 \rangle$, but $\sigma$ in general will not. We replace $\sigma\tau$
with $\sigma\alpha_1 \alpha_2 \tau$ where the cyclic spike automorphism $\alpha_i$ denotes 
conjugating $\langle M_i \rangle$ with $h^{-1}$ and fixing the other two factors. Then 
$\sigma\alpha_1 \alpha_2 \tau$ has the same effect on $H$ as $\sigma\tau$ while being the identity on
$\langle M_0 \rangle\star\langle M_1 \rangle$.
\end{proof}

\end{document}